\documentclass [14pt]{amsart}

\usepackage[T1]{fontenc}
\usepackage{ae,aecompl}
\usepackage{xcolor}
\usepackage{scalerel,stackengine}
\usepackage{mathabx}
\usepackage{amssymb,amscd}
\usepackage{tcolorbox}
\usepackage{tikz-cd}
\usepackage{hyperref}

\DeclareMathAlphabet{\mathbbu}{U}{bbold}{m}{n}

\usepackage{graphicx}
\allowdisplaybreaks

\renewcommand{\geq}{\geqslant}

\theoremstyle{bold}
\newtheorem{theorem}{Theorem}[section]
\newtheorem{proposition}[theorem]{Proposition}
\newtheorem{lemma}[theorem]{Lemma}
\newtheorem{corollary}[theorem]{Corollary}

\newtheorem{fact}[theorem]{Fact}

\theoremstyle{definition}
\newtheorem{definition}[theorem]{Definition}
\newtheorem{remark}[theorem]{Remark}
\newtheorem{example}[theorem]{Example}

\DeclareMathOperator{\card}{\mbox{\rm card}}
\DeclareMathOperator{\supp}{\mbox{\rm supp}}
\DeclareMathOperator{\dens}{\mbox{\rm dens}}
\DeclareMathOperator{\Span}{\mbox{\rm span}}
\DeclareMathOperator{\udens}{\overline{\mbox{\rm dens}}}
\DeclareMathOperator{\dd}{\mbox{d}}
\DeclareMathOperator{\sinc}{\mbox{\rm sinc}}

\newcommand{\cA}{\mathcal{A}}
\newcommand{\cB}{\mathcal{B}}

\newcommand{\cL}{\mathcal{L}}
\newcommand{\cM}{\mathcal{M}}

\newcommand{\Fib}{\mbox{\textbf{Fib}}}

\newcommand{\vol}{\mbox{vol}}

\newcommand{\NN}{\mathbb{N}}

\newcommand{\ZZ}{\mathbb{Z}}
\newcommand{\QQ}{\mathbb{Q}}
\newcommand{\RR}{\mathbb{R}}
\newcommand{\CC}{\mathbb{C}}

\newcommand{\KK}{\mathbb{K}}

\newcommand{\eps}{\varepsilon}

\newcommand{\oplam}{\mbox{\Large $\curlywedge$}}
\newcommand{\hoplam}{\mbox{\Large $\curlywedge^\star$}}

\stackMath
\newcommand\reallywidehat[1]{%
\savestack{\tmpbox}{\stretchto{%
  \scaleto{%
    \scalerel*[\widthof{\ensuremath{#1}}]{\kern-.6pt\bigwedge\kern-.6pt}%
    {\rule[-\textheight/2]{1ex}{\textheight}}
  }{\textheight}%
}{0.5ex}}%
\stackon[1pt]{#1}{\tmpbox}%
}

\newcommand\reallywidecheck[1]{%
\savestack{\tmpbox}{\stretchto{%
  \scaleto{%
    \scalerel*[\widthof{\ensuremath{#1}}]{\kern-.6pt\bigwedge\kern-.6pt}%
    {\rule[-\textheight/2]{1ex}{\textheight}}
  }{\textheight}%
}{0.5ex}}%
\stackon[1pt]{#1}{\scalebox{-1}{\tmpbox}}%
}

\newcommand{\Cu}{C_{\mathsf{u}}}
\newcommand{\Cc}{C_{\mathsf{c}}}

\begin{document}
\title{Why do (weak) Meyer sets diffract?}

\dedicatory{We dedicate this work to Daniel Lenz on the occasion of his $50^{th}$ birthday.}

\author{Nicolae Strungaru}
\address{Department of Mathematical Sciences, MacEwan University \\
10700 -- 104 Avenue, Edmonton, AB, T5J 4S2, Canada\\
and \\
Institute of Mathematics ``Simon Stoilow''\\
Bucharest, Romania}
\email{strungarun@macewan.ca}
\urladdr{http://academic.macewan.ca/strungarun/}

\begin{abstract} Given a weak model set in a locally compact Abelian, group we construct a relatively dense set of common Bragg peaks for all its subsets that have non-trivial Bragg spectrum. Next, we construct a relatively dense set of common norm almost periods for the diffraction, pure point, absolutely continuous and singular continuous spectrum, respectively, of all its subsets. We use the Fibonacci model set to illustrate these phenomena.  We extend all these results to arbitrary translation bounded weighted Dirac combs supported within some Meyer set. We complete the paper by discussing extensions of the existence of the generalized Eberlein decomposition for measures supported within some Meyer set.
\end{abstract}

\maketitle

\section{Introduction}

Physical quasicrystals were discovered in the 1980's by Dan Shechtman \cite{She}. Shechtman's discovery forced the International Union of Crystallography to redefine a crystal to be ``any solid having an essentially discrete diffraction diagram” \cite{IUCR}. While the word ``essentially" is vague in this context, it is usually understood to mean a relatively dense set of Bragg peaks.

The largest class of mathematical models which is easy to classify, and has a relatively dense set of Bragg peaks, is the class of Meyer sets \cite{NS1}. Meyer sets were introduced by Y. Meyer \cite{Meyer} in the 1970's as approximate solutions to an infinite system of linear equations in $\RR/\ZZ$. They are produced via a cut-and-project scheme as relatively dense subsets of (regular) model sets, and can be characterised via harmonic, discrete geometry and algebraic properties \cite{LAG1,Meyer,MOO,Moody,NS11}.
Their connection to aperiodic order was observed in the 1990 by Lagarias \cite{LAG1,LAG} and Moody \cite{MOO,Moody}. In recent years, many properties about their diffraction have been proven \cite{NS1,NS2,NS5,NS11,NS20a}. All these results can be traced back to the long-range order of the lattice $\cL$ in the cut-and-project scheme, and to the compactness of a covering window. Physicists have also shown interest in these results about the spectra of Meyer sets (see \cite{BCG,GQ} for example).

It was shown by Hof \cite{Hof2,HOF3} that Euclidean model sets with nice windows have a clear diffraction pattern (pure point). The result was generalized to regular model sets in second countable LCAG by Schlottmann \cite{Martin2}, via the usage of dynamical systems and the Dworkin argument (see \cite{BL,Dwo,LS} for background). Baake and Moody gave an alternate proof of this result via almost periodicity \cite{BM} (compare \cite{LR}) and emphasized the deep connection between almost periodicity and long range order. In \cite{CRS,CRS2}, we showed that the diffraction formula for regular model sets is just the Poisson Summation Formula (PSF) for the underlying lattice, applied twice. This result shows explicitly how the clear diffraction diagram of a regular model set is just a residue of the strong order of the lattice in the CPS. Recently, the pure point diffraction of model sets was shown via the almost periodicity of the Dirac comb of the model set \cite{LSS,LSS2,MEY} (compare \cite{LAG}).

As subsets of highly ordered sets, Meyer sets should show some signs of long range order, and they indeed do. They show a relatively dense Bragg spectrum \cite{NS1,NS2,NS11}, which is highly ordered \cite{NS2,NS11}. The continuous diffraction spectrum is either empty or has a relatively dense support \cite{NS1,NS5,NS20a}. Moreover, each of the absolutely continuous and singular continuous spectra is a norm almost periodic measure \cite{NS20a}. Some of these results are important in the study of Pisot substitutions (see for example \cite{BaGa,BG2}).

All these results have been proved by two different techniques:
\begin{itemize}
  \item{} by using harmonic properties of a covering model set;
  \item{} by using the relative denseness of the sets $\Lambda^\eps$ of $\eps$-dual characters.
\end{itemize}
In fact, it turns out that these approaches are not that different. Indeed, if we know a covering model set, we can explicitly write out a relatively dense set of $\eps$-dual characters, which only depends on the covering model set.

The main goal of this paper is to reanalyze the diffraction spectra of a Meyer set via a study of harmonic properties of a covering model set $\oplam(W)$. This allows us prove diffraction results for Meyer sets, that are stronger than the earlier work in this direction \cite{NS1,NS2,NS5,NS11,NS20a}. Moreover, the results are uniform in the following sense:
\begin{itemize}
  \item{} We can find a relatively dense set of common Bragg peaks for all the Meyer sets with non-trivial Bragg spectrum, which are subsets of the same model set.
  \item{} For each $\eps >0$, we can find a relatively dense set $P_\eps \subseteq G$, which depends on the covering model set only, such that, for every Meyer subset $\Lambda \subseteq \oplam(W)$, the elements of $P_\eps$ are norm almost periods for the diffraction spectrum, pure point diffraction spectrum, continuous diffraction spectrum, and the absolutely continuous diffraction spectrum of $\Lambda$.
\end{itemize}
The relatively dense set of common Bragg peaks can be calculated explicitly in terms of $\oplam(W)$. Furthermore, the sets $P_\eps$ can be calculated explicitly in terms of $\oplam(W)$ and $\eps>0$. To emphasize this phenomena, we calculate these sets explicitly for the Fibonacci model set.

It is worth mentioning that the proofs in this paper are simpler than the original proofs of \cite{NS1,NS2,NS5,NS11,NS20a}, and do not rely on the relatively denseness of Meyer sets. Therefore, while relative denseness plays a crucial role in the classification of Meyer sets \cite{BLM,LAG1,Meyer,Moody,NS11}, it is not important in the current paper. This allows us generalize the previous results to the larger class of subsets of model sets with non-trivial Bragg peak at the origin. For this reason, in this paper we study weak Meyer sets, meaning arbitrary subsets of (weak) model sets.

\smallskip
The paper is organized as follows. We start by using the Fibonacci model set, and its subsets, as motivational example in Section~\ref{sect:fib}. Next, given a weak model set $\oplam(W)$, we construct in Theorem~\ref{thm bragg meyer set} a relatively dense set $B \subseteq \widehat{G}$, such that all subsets of $\oplam(W)$ with non-trivial Bragg spectrum have Bragg peaks at all points in $B$. We furthermore give a lower bound for the intensity of the Bragg peak at any $\chi \in B$, in terms of the Bragg peak at 0; a lower bound which only depends on the covering model set $\oplam(W)$. This lower bound can be then used to construct an explicit set of high intensity Bragg peaks, simultaneously for all the subsets $\Lambda \subseteq \oplam(W)$ with non-trivial Bragg spectrum.

Next, we review and briefly improve the ping-pong lemma for Meyer sets (Theorem~\ref{ping pong theorem}). This result is the key to the study of the continuous component of the diffraction spectrum of weak Meyer sets. The ping-pong lemma can also be used to study the Bragg, absolutely continuous and singular continuous component of weighted complex combs with weak Meyer set support. By using the ping-pong lemma, we show that for every weak model set $\oplam(W)$, each compact $K\subseteq \widehat{G}$, and each $\eps$, there exists a relatively dense set $P_{\eps}(\oplam(W),K)$ of points which are $\eps$ - $\| \cdot \|_{K}$ norm almost periods for the diffraction, pure point, absolutely continuous and singular continuous diffraction component for every weak Meyer set $\Lambda \subseteq \oplam(W)$.

We complete the paper by generalizing the existence of the generalized Eberlein decomposition for measures with weak Meyer set support to a much larger class of decompositions. We introduce the notion of FCDM functions (see Definition~\ref{def FCDM}), and show that for every Fourier transformable measure $\gamma$ supported inside a weak model set $\oplam(W)$ and every FCDM function $P: \cM^\infty(\widehat{G}) \to \cM^\infty(\widehat{G})$, there exists a Fourier transformable measure $\gamma_P$ supported inside $\oplam(W)$ such that $\widehat{\gamma_P}=P(\widehat{\gamma})$.

Since the projections on the pure point, continuous, absolutely continuous, and singular continuous component are FCDM functions, the existence of the generalized Eberlein decomposition is an immediate consequence of this result.

\section{Fibonacci model set}\label{sect:fib}

We start by looking at a well known example, which we use as motivation for the remaining of the paper. Let us start by briefly recalling the Fibonacci model set, and refer the reader to \cite[Chapter~7]{TAO} for more details.

Consider
\begin{displaymath}
\cL:=\left\{ \begin{bmatrix}m+n \tau \\ m+n\tau'\end{bmatrix} : m,n \in \ZZ \right\}= \ZZ \begin{bmatrix} 1 \\ 1 \end{bmatrix} \oplus \ZZ \begin{bmatrix} \tau \\ \tau' \end{bmatrix}  \subseteq \RR^2  \,,
\end{displaymath}
where $\tau=\frac{1+\sqrt{5}}{2}$ and $\tau'=\frac{1-\sqrt{5}}{2}$. This is a lattice in $\RR \times \RR$, with dual lattice
\begin{equation}\label{eq:dual fib lattice}
\cL^0= \frac{1}{\sqrt{5}}\ZZ \begin{bmatrix} -\tau'  \\ \tau  \end{bmatrix} \oplus \frac{1}{\sqrt{5}} \ZZ \begin{bmatrix} 1 \\ -1 \end{bmatrix}=  \left\{ \begin{bmatrix} \frac{n-m\tau'}{\sqrt{5}} \\ \frac{m \tau -n}{\sqrt{5}} \end{bmatrix} : m,n \in \ZZ \right\}  \,.
\end{equation}
Let $\star$ denote the restriction to $\ZZ[\tau]$ of the Galois conjugation of the field $\QQ(\tau)$, that is
\[
(m+n \tau)^\star=m+n \tau' \,.
\]
We can now recall the following definition:
\begin{definition} If $I \subseteq \RR$ is any bounded interval we will denote
\[
\oplam(I) = \{ x : x \in \ZZ[\tau], x^\star \in I \} \,.
\]
The set $\oplam([-1, \tau-1))$ is called the \textbf{Fibonacci model set} and is simply denoted by $\Fib$.
\end{definition}
Given an interval $J \subseteq \RR$ we also denote
\begin{eqnarray*}
  \hoplam(J) &=& \left\{ \frac{n-m\tau'}{\sqrt{5}} : m,n \in \ZZ, \frac{m \tau -n}{\sqrt{5}}  \in J \right\} \,.
\end{eqnarray*}
Here, we use the $\hoplam$ to emphasize that this model set uses the dual lattice $\cL^0$.

\subsection{Bragg spectrum of subsets of the Fibonacci model set}

The critical property relating these models sets is the following.

\begin{proposition}\label{prop epsilon dual} Let $a,b \in (0, \infty)$. If $x \in \oplam([-a,a])$ and $y \in \hoplam([-b,b])$ then
\begin{displaymath}
\left| e^{-2 \pi  i x \cdot y } -1 \right| \leq 2 \pi ab \,.
\end{displaymath}
\end{proposition}
\begin{proof}
Since $(x,x^\star) \in \cL$ and $(y,y^\star) \in \cL^0$ we have by \eqref{eq:dual fib lattice} that
\begin{displaymath}
e^{-2 \pi  i (x\cdot y+x^\star\cdot y^\star) } =1 \,.
\end{displaymath}
It follows that
\begin{align*}
\left| e^{-2 \pi  i x \cdot y } -1 \right| &= \left| e^{-2 \pi  i x \cdot y } -e^{-2 \pi  i (x \cdot y+x^\star\cdot y^\star) } \right|=\left|1 -e^{-2 \pi  i (x^\star \cdot y^\star) } \right| \\
 = 2 \left|\sin( \pi   (x^\star \cdot y^\star))\right| &\leq 2 \left|\pi   (x^\star \cdot y^\star)\right| \leq 2 \pi a b \,.
\end{align*}
\end{proof}

\medskip

Next we find some rough estimates for $R=R(a)$ such that $\hoplam((-a,a))$ is $R$-relatively dense.

\begin{lemma}\label{hoplam rel dense} Let $0 < a< \frac{\tau}{\sqrt{5}}$. Then,
$$
\hoplam ((-a, a))  +  [0 , \frac{\tau^3}{5a} ] = \RR \,.
$$
\end{lemma}
\begin{proof}
It is easy to see that
the rectangle $[0 , \frac{\tau}{\sqrt{5}}] \times [-\frac{1}{\sqrt{5}}, \frac{\tau}{\sqrt{5}}]$ contains a fundamental domain of $\cL$.
This immediately implies that
\[
\hoplam ((-\frac{\tau}{\sqrt{5}}, \frac{\tau}{\sqrt{5}}))  +  [0 , \frac{\tau}{\sqrt{5}}] = \RR \,.
\]
Next, since $\tau$ is a unit in $\ZZ[\tau]$, a trivial computation (compare \cite[Fact~3.5]{KST}) shows that for all $\alpha >0$ we have
\[
\tau \hoplam((-\alpha,\alpha))= \hoplam((-\frac{\alpha}{\tau}, \frac{\alpha}{\tau})) \,.
\]

Now, since $a < \frac{\tau}{\sqrt{5}}$, picking the largest $n \in \NN$ so that $a \tau^n < \frac{\tau}{\sqrt{5}}$ we get
\begin{align*}
\hoplam ((-a, a))  +  [0 , \frac{\tau^3}{5a} ] & \supseteq \hoplam((-\frac{\frac{\tau}{\sqrt{5}}}{\tau^{n+1}},\frac{\frac{\tau}{\sqrt{5}}}{\tau^{n+1}} ))+[0, \frac{\tau^{n+2}}{\sqrt{5}}] \\
&= \tau^{n+1} \left(\hoplam ((-\frac{\tau}{\sqrt{5}}, \frac{\tau}{\sqrt{5}}))+[0 , \frac{\tau}{\sqrt{5}}]  \right)=\RR \,.
\end{align*}
This completes the proof.
\end{proof}

\smallskip
We are now ready to give results about the pure point component of the diffraction spectrum for a subset of $\Fib$.
We show that the structure of the pure point component for subsets of Fibonacci is somewhat similar in nature to the structure of the pure point
component for subsets of lattices \cite{Ba}. We will expand on this similarity in Section~\ref{sect:diss}.

\smallskip
Before looking at the diffraction, let us briefly recall that for a subset $\Lambda$ of the Fibonacci model set, any cluster point $\gamma_\Lambda$ of
\[
\frac{1}{\vol (B_{R}(0))} \delta_{\Lambda \cap B_{R}(0)}*\widetilde{\delta_{\Lambda \cap B_{R}(0)}}
\]
is called an \textbf{autocorrelation} measure of $\Lambda$. Such cluster points always exist (see for example \cite[Prop.~9.1]{TAO}).

Given any autocorrelation $\gamma_{\Lambda}$, there exists a unique positive measure $\widehat{\gamma_{\Lambda}}$ on $\RR$ such that, for all $\varphi \in \Cc(\RR)$ we have \cite[Thm.~4.5]{BF}
\[
\int_{\RR} \varphi*\tilde{\varphi}(t) \dd \gamma_{\Lambda}(t) = \int_{\RR} \left| \check{\varphi}(s) \right|^2 \dd \widehat{\gamma_{\Lambda}}(s) \,.
\]
The measure $\widehat{\gamma_{\Lambda}}$ is called a \textbf{diffraction measure of $\Lambda$}, and can also be seen as the distributional Fourier transform of the tempered measure $\gamma_{\Lambda}$.

\smallskip

We can now prove the following result.

\begin{proposition} \label{prop intensity bragg peaks fib} Let $\Lambda$ be an arbitrary subset of the Fibonacci model set, and let $a>0$ be a real number. Assume that
\begin{equation}\label{eq:I>0}
I := \widehat{ \gamma_{\Lambda}}(\{ 0\}) >0 \,.
\end{equation}
Then, for all $k \in \hoplam([-a,a])$ we have
\begin{displaymath}
\widehat{\gamma_\Lambda}(\{k\}) \geq \left(1- 2 \pi a \tau \right) I \,.
\end{displaymath}
\end{proposition}
\begin{proof}
Note that by construction we have
\begin{displaymath}
\supp(\gamma_\Lambda) \subseteq \oplam([-\tau, \tau] )=: \Delta  \,.
\end{displaymath}

Next, by \cite[Theorem~4.10.14]{MoSt} we have
\begin{align*}
I&=\widehat{\gamma_\Lambda}(\{ 0\})= \lim_n \frac{1}{2n} \left( \sum_{x \in \Delta \cap [-n,n]} \gamma_{\Lambda}(\{x\}) \right) \\
\widehat{\gamma_\Lambda}(\{ k\})&= \lim_n \frac{1}{2n} \left( \sum_{x \in \Delta \cap [-n,n]} e^{-2 \pi i x \cdot k}  \gamma_{\Lambda}(\{x\})\right) \,.
\end{align*}
Therefore,  (compare  \cite[Proposition~5.9.8]{NS11} and \cite[Theorem~3.1]{NS2}) we have
\begin{align*}
\widehat{\gamma_\Lambda}(\{ 0\})-\widehat{\gamma_\Lambda}(\{ k\})&= \left|\widehat{\gamma_\Lambda}(\{ 0\})-\widehat{\gamma_\Lambda}(\{ k\}) \right|\\
&\leqslant \limsup_n \frac{1}{2n}\left( \sum_{x \in \Delta \cap [-n,n]} \left| 1- e^{-2 \pi i x \cdot k} \right| \gamma_{\Lambda}(\{x\}) \right) \\
& \leqslant (2 \pi a \tau) \limsup_n \frac{1}{2n}\left( \sum_{x \in \Delta \cap [-n,n]}  \gamma_{\Lambda}(\{x\}) \right)= (2 \pi a \tau)I \,,
\end{align*}
with the first inequality on the last line following from Prop.~\ref{prop epsilon dual}.
\end{proof}

Before looking at some consequences, let us briefly discuss the condition $I>0$ and its potential connection to positive density.

\begin{remark} Let $\Lambda \subseteq \Fib$ and let $\gamma_{\Lambda}$ be an autocorrelation of $\Lambda$. Let $R_n \to \infty$ be an increasing sequence, such that $\gamma_{\Lambda}$ is a limit along $R_n$.

Then, the limit
\[
\gamma_{\Lambda}(\{0 \}) = \lim_n \frac{1}{\vol (B_{R_n}(0))}\card ( \Lambda \cap B_{R_n}(0))=: \dens_{R_n}(\Lambda)
\]
exists. Moreover, \eqref{eq:I>0} implies that $\dens_{R_n}(\Lambda) >0$.

If $\Lambda$ satisfies the consistent phase property (see \cite{LSS} for details), then \eqref{eq:I>0} is equivalent to $\dens_{R_n}(\Lambda) >0$, but in general this does not seem to be the case.

Finally, if $\Lambda$ is relatively dense, then a simple computation shows that
\[
\gamma_{\Lambda}(\{0 \}) \geq \left( \liminf_{n} \inf_{t \in \RR} \frac{1}{\vol (B_{R_n}(t))}\card ( \Lambda \cap B_{R_n}(t)) \right)^2 >0 \,.
\]
In particular, \eqref{eq:I>0} trivially holds for Meyer subsets of $\Fib$.
\end{remark}

\smallskip

Prop.~\ref{prop intensity bragg peaks fib} has the following immediate consequence.

\begin{theorem}\label{them:1} Let $\Lambda \subseteq \Fib$ be any subset satisfying \eqref{eq:I>0}. Then, $\Lambda$ has a Bragg peak at every $k \in \hoplam((-\frac{1}{2 \pi \tau},\frac{1}{2 \pi \tau} ))$, of  intensity
  \begin{displaymath}
  \widehat{\gamma_\Lambda}(\{ k\}) \geq \left(1- 2 \pi |k^\star| \tau \right) I >0 \,.
  \end{displaymath}
\end{theorem}
\begin{proof} Let $k \in \ZZ[\tau]$. Set $a:=|k^\star|$. Then $k \in \hoplam([-a,a])$ and hence by Proposition~\ref{prop intensity bragg peaks fib} we have
\begin{displaymath}
\widehat{\gamma_\Lambda}(\{ k\}) \geq\left(1- 2 \pi a \tau \right) I= \left(1- 2 \pi |k^\star| \tau \right) I >0  \,.
\end{displaymath}
\end{proof}

\begin{corollary}[High intensity Bragg peaks]\label{high int} Let $0 < \eps <1$, and let
$$P_\eps:= \hoplam([-\frac{\eps}{2 \pi \tau},\frac{\eps}{2 \pi \tau} ]) \,.$$
Then, every subset $\Lambda \subseteq \Fib$ satisfying \eqref{eq:I>0} has a Bragg peak at every $k \in P_\eps$, of intensity at least
$$
\widehat{\gamma_\Lambda}(\{k\}) \geq (1-\eps) I \,.
$$
\end{corollary}
\begin{proof}
By Proposition~\ref{prop intensity bragg peaks fib} we have
\begin{displaymath}
\widehat{\gamma_\Lambda}(\{ k\}) \geq \left(1- 2 \pi |k^\star| \tau \right) I \geq (1-\eps) I \,.
\end{displaymath}
\end{proof}

\begin{figure}[ht]\label{fib1}
\includegraphics[width=7cm]{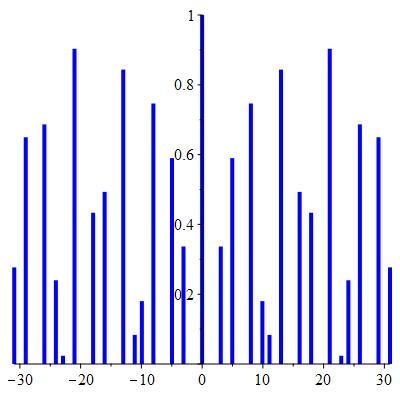}
\centering
\caption{All subsets of $\Fib$ satisfying \eqref{eq:I>0}, have Bragg peaks at the depicted positions. The intensities are at least the height of the vertical line, drawn as a percentage of the intensity of Bragg peak at the origin. Picture created in Maple by Anna Klick.}
\end{figure}

\subsection{Continuous spectrum of subsets of Fibonacci}

We next look at the continuous spectrum component of subsets of $\Fib$. Since the details are long and technical, yet simple, we skip most steps and refer the reader
to the extended arXiv version of this paper \cite{NS21}.

\medskip
First, define
\begin{displaymath}
h(x)=\frac{27}{8} 1_{[-\tau-1, \tau+1]}* 1_{[-\frac{1}{3},\frac{1}{3}]}*  1_{[-\frac{1}{3},\frac{1}{3}]}*  1_{[-\frac{1}{3},\frac{1}{3}]}
\end{displaymath}
and define
\begin{equation}\label{eq:pp-fib}
\omega:= \sum_{x \in \ZZ[\tau]} h(x^\star) \delta_x \,.
\end{equation}
Then, it is easy to check that $\omega$ is a ping-pong measure for the Fibonacci model set (see Def.~\ref{def:ping-pong} for definition and Prop.~\ref{existence ping pong} (b) below for proofs). Then, Prop.~\ref{existence ping pong} and Thm.~\ref{ping pong theorem} below give:

\begin{theorem}\label{theorem:fib ping pong} Let $\omega$ be as in \eqref{eq:pp-fib}. Let
 $\varphi \in \Cc(\RR)$ be such that $\supp(\varphi) \subseteq (-\frac{1}{4}, \frac{1}{4})$ and $\| \varphi \|_2=1$. Then,
\begin{itemize}
  \item[(a)] $\omega$ is Fourier transformable, and
  \begin{displaymath}
\widehat{\omega}=\frac{3+\sqrt{5}}{\sqrt{5}} \sum_{(y,y^\star) \in \cL^0} \sinc (2\pi (\tau+1) y^\star) \sinc^3 (\frac{2\pi}{3} y^\star) \delta_y \,.
  \end{displaymath}
  \item[(b)] For each $\Lambda \subseteq \Fib$, the measure $\nu:= \left|\widecheck{\varphi}\right|^2 \widehat{\gamma_\Lambda}$ is finite and
  \begin{displaymath}
\widehat{\gamma_\Lambda}= \widehat{\omega} * \nu \,.
  \end{displaymath}\qed
  \end{itemize}
\end{theorem}

As an immediate consequence we get:

\begin{corollary} Under the assumptions of Theorem~\ref{theorem:fib ping pong}(b), we have
\[
  \left(\widehat{\gamma_\Lambda}\right)_{\mathsf{pp}}  =  \widehat{\omega} * \nu_{\mathsf{pp}} \,;\,
  \left(\widehat{\gamma_\Lambda}\right)_{\mathsf{ac}}  =  \widehat{\omega} * \nu_{\mathsf{ac}} \,;\,
 \left(\widehat{\gamma_\Lambda}\right)_{\mathsf{sc}}  =  \widehat{\omega} * \nu_{\mathsf{sc}} \,.
\]\qed
\end{corollary}

Now, we proceed to get quantitative results about the norm almost periods of $\widehat{\omega}$. First, it is easy to see (\cite[Lemma~A.3]{NS21}) that, in the Fibonacci dual CPS $(\RR, \RR, \cL^0)$, any translate of $[-\frac{1}{2}, \frac{3}{2}] \times [0, \frac{1}{20}]$ meets $\cL^0$ at most once. This, combined with the fact that $g(x)=\sinc(ax)\sinc^3(bx)$ is a Lipshitz function, with Lipshitz constant $\frac{a+3b}{2}$  (\cite[Lemma~A.4]{NS21}),
and a standard computation (\cite[Theorem 2.14]{NS21}) yields:

\begin{theorem}\label{fib2}  Let $\omega$ be as in \eqref{eq:pp-fib} and let $0 < \alpha <\frac{1}{81}$ . Then, for all $t \in \hoplam([-\alpha, \alpha])$, we have
\[
\|T_t \widehat{\omega}-  \widehat{\omega} \|_{[-1,1]} < 2507 (\alpha)^{\frac{3}{4}} \,.
\] \qed
\end{theorem}
Let us note here that, the statement in \cite[Theorem 2.14]{NS21} uses the norm $\| \, \|_{[-\frac{1}{2}, \frac{3}{2}]}$. Anyhow, it follows trivially from the definition that, a compact set and any of its translates define equal norms, which allow us use $\| \, \|_{[-1,1]}$ instead.

\smallskip

We can now use Theorem~\ref{theorem:fib ping pong} to relate the almost periods of the diffraction of any subset of $\Fib$ to the almost periods of $\widehat{\omega}$. Indeed, a simple computation
yields (\cite[Theorem 2.15]{NS21}).

\begin{theorem}\label{fib 3}
Let $\omega$ be as in \eqref{eq:pp-fib} and let $\Lambda \subseteq \Fib$ be any set. Let $\gamma$ be an autocorrelation of $\Lambda$, calculated along some $\cA=\{ B_{R_n}(0) \}$.

Then, for all $t \in \RR$ we have
\[
\| T_t  \left(\widehat{\gamma}\right) -  \left(\widehat{\gamma}\right) \|_{[0,1]} \leq {\rm d} \| T_t \widehat{\omega}-\widehat{\omega} \|_{[-1, 1]} \,,
\]
where
$$
{\rm d}=\gamma(\{ 0\})= \dens_{\cA}(\Lambda) \,.
$$
In particular
\[
\| T_t  \left(\widehat{\gamma}\right) -  \left(\widehat{\gamma}\right) \|_{[0,1]} \leq  \frac{\tau}{\sqrt{5}}\| T_t \widehat{\omega}-\widehat{\omega} \|_{[-1, 1]} \,.
\]
\end{theorem}\qed

By combining Theorem~\ref{fib2} with Theorem~\ref{fib 3} we get (\cite[Theorem 2.15]{NS21}):

\begin{theorem}\label{fib support} Let $0< \eps < 1$ and $t \in \hoplam([-\frac{4\eps^{\frac{4}{3}} }{10^5}, \frac{4\eps^{\frac{4}{3}} }{10^5}])$. Then, for all $\Lambda \subseteq \Fib$ and all autocorrelations $\gamma$ of $\Lambda$, we have
\begin{align*}
&\max \left\{
\| T_t  \left(\widehat{\gamma}\right) -  \left(\widehat{\gamma}\right) \|_{[0,1]} ; \| T_t  \left(\widehat{\gamma}\right)_{\mathsf{pp}} -  \left(\widehat{\gamma}\right)_{\mathsf{pp}} \|_{[0,1]}  ; \right. \\
&\left. \| T_t  \left(\widehat{\gamma}\right)_{\mathsf{ac}} -  \left(\widehat{\gamma}\right)_{\mathsf{ac}} \|_{[0,1]} ;
\| T_t  \left(\widehat{\gamma}\right)_{\mathsf{sc}} -  \left(\widehat{\gamma}\right)_{\mathsf{sc}} \|_{[0,1]} \right\} < \eps
\end{align*}
In particular, $\widehat{\gamma}, \left(\widehat{\gamma}\right)_{\mathsf{pp}},  \left(\widehat{\gamma}\right)_{\mathsf{ac}}$ and $\left(\widehat{\gamma}\right)_{\mathsf{sc}} $ are norm almost periodic measures. \qed
\end{theorem}

\section{Pure point spectrum of weak Meyer sets}\label{sect pure point}

Throughout this paper, $G$ denotes a second countable locally compact Abelian group (LCAG). We should note here in passing that metrisability of $G$ is only used in an implicit way whenever we refer to the autocorrelation of a point set (or measure). In the literature, the autocorrelation is defined as the limit along a subsequence of our averaging sequence, and the existence of such subsequences relies on the second countability of $G$.
One can easily get around this issue by working with a subnet of the averaging sequence, or even more generally by starting with a van Hove net (in which case it is likely that we can also drop the $\sigma$-compactness assumption).
Anyhow, since the basic theory of diffraction using van Hove nets is not setup yet, and there is no good reference in this direction, we prefer to restrict to second countable LCAG's for this project.

We will assume that the reader is familiar with cut-and-project schemes, Delone sets and Meyer sets, Radon measures, total variation of a measure, convolution between functions and measures, van Hove sequences, autocorrelation and diffraction, norm almost periodicity, and refer the unfamiliar reader to \cite{TAO,NS11,NS20a}.

In this paper we are interested in translation bounded measures supported inside model sets. Let us briefly recall that a measure $\mu$ is called translation bounded if, for all compact sets $K \subseteq G$ we have
\[
\| \mu \|_{K}:= \sup_{t \in G} \left| \mu \right|(t+K) < \infty \,.
\]
Here, $\left| \mu \right|$ denotes the total variation measure of $\mu$ (see \cite[Page 252]{Ped} for the definition). As usual, we denote by $\cM^\infty(G)$ the space of translation bounded measures.

\smallskip

As we mentioned in the introduction, the goal of this paper is to describe properties of the spectrum of Meyer sets in terms of covering model sets, and that relative denseness of Meyer sets can be replaced by weaker properties. Because of this, we introduce the following definition.

\begin{definition} A set\/ $\Lambda$ in a LCAG\/ $G$ is called a \textbf{weak Meyer set} if it is a subset of a model set.
\end{definition}

It is easy to see that a set\/ $\Lambda$ is a weak Meyer set if and only if it is a subset of a Meyer set. Moreover, any subset of a weak Meyer set is a weak Meyer set.
Finally, a weak Meyer set $\Lambda \subseteq G$ is a Meyer set if and only if\/ $\Lambda$ is relatively dense.

The maximal density weak model sets studied in \cite{BHS,KR,KR2} as well as the weak model sets with precompact Borel windows \cite{KRS} are weak Meyer sets, which are not necessarily Meyer sets. Any subset of a weak model set is a weak Meyer set. $\cB$-free systems and their subsets \cite{DKK,GK,KKL,KLRS} also provide many interesting examples of weak Meyer sets.

In fact, any weak Meyer set $\Lambda$ is a subset of a model set, and the cut and project scheme can be chosen such that the $\star$-mapping is one to one. As $G$ is $\sigma$-compact, both $\Lambda$ and the covering model set $\oplam(W)$ are countable, and hence so are their images under the $\star$-mapping. By simply picking $\Lambda^\star$ as the window, or by starting with $W$ and eliminating the extra points in $\oplam(W) \backslash \Lambda$, one can construct a Borel window $W'$ which gives $\Lambda =\oplam(W')$ as a weak model set. This shows that the notions of weak Meyer sets and weak model sets coincide.

\medskip

Let us briefly remind the reader here the concept of a cut-and-project scheme. A cut-and-project scheme (or simply CPS) is a triple $(G, H, \cL)$ consisting of our group $G$, a LCAG $H$, and a lattice $\cL \subseteq G \times H$, with the following two extra properties:
\begin{itemize}
  \item{} The restriction $\pi_G|_{\cL}$ of the projection $\pi_{G} : G \times H \to G$ to $\cL$ is one to one.
  \item{} $\pi_{H}(\cL)$ is dense in $H$.
\end{itemize}
Given a cut-and-project scheme $(G, H, \cL)$, and a set $W \subseteq H$, we denote
\[
\oplam(W):= \{ \pi_G(l): l \in \cL, \pi_{H}(l) \in W \} \,.
\]

\smallskip
Next, let us briefly recall that the dual group $\widehat{G}$ is defined as the group of continuous homomorphisms $\chi : G \to U(1):=\{z \in \CC : |z| =1 \}$.
This becomes a LCAG under a topology for which the sets
\begin{equation}\label{dualbas}
N(K, \eps) := \{ \chi \in \widehat{G} : \left| \chi(x)-1 \right| <1 \forall x \in K \} \,
\end{equation}
with $\eps >0$ and compact $K \subseteq G$, form a basis of open sets at $0 \in \widehat{G}$ (see \cite[Sect.~1.2]{RUD} for details).

\smallskip
Given a cut-and-project scheme $(G, H, \cL)$, the annihilator $\cL^0$ in $\widehat{G \times H} = \widehat{G} \times \widehat{H}$ of $\cL$ is a lattice. Moreover, $(\widehat{G}, \widehat{H}, \cL^0)$ is a cut-and-project scheme \cite{Moody}. We refer to this as the dual cut-and-project scheme. To emphasize that we work in this dual space, for $W \subseteq \widehat{H}$, we will use the notation
\[
\hoplam(W):= \{ \pi_{\widehat{G}}(w): w \in \cL^0, \pi_{\widehat{H}}(w) \in W \} \,.
\]

\medskip

To make things easier to follow, we will use the following setting.

\begin{tcolorbox}
For this entire section, $(G,H, \cL)$ is a fixed CPS and $W \subseteq H$ is a fixed compact set. For $\eps >0$, we set
\begin{align*}
W_\eps&:= N(W-W, \eps) \,,\\
B_\eps &:= \hoplam(W_\eps) \,.
\end{align*}
Moreover, for each $(\chi, \chi^\star) \in \cL^0$ we set
\[
s(\chi):= \sup \left\{ \left| \chi^\star(t) -1 \right| : t \in W-W \right\} \,.
\]
\end{tcolorbox}
Note that since $W$ is compact, so is $W-W$. Therefore, the continuous function $\left| \chi^\star (t) -1 \right|$ attains its maximum on $W-W$. In particular,
for all $\chi \in B_\eps$, we have
\begin{equation}\label{eqschi}
s(\chi) < \eps \,.
\end{equation}

\begin{remark}
\begin{itemize}
  \item[(a)]Since $W_\eps$ is open, $B_\eps$ is relatively dense \cite{MOO}.
\item[(b)] If $0< \eps < \eps'$, then $B_{\eps} \subseteq B_{\eps'}$.
\end{itemize}
\end{remark}

Similarly to subsets of $\Fib$, we can prove that any weak Meyer subset $\Lambda \subseteq \oplam(W)$ with non-trivial Bragg spectrum, has a Bragg peak at each $\chi \in B_1$. Moreover, for all $\eps>0$ small enough, $\Lambda$ has a high intensity Bragg peak at each $\chi \in B_\eps$.

\begin{theorem}\label{thm bragg meyer set} Let $\Lambda \subseteq \oplam(W)$ and let $\gamma_\Lambda$ be an autocorrelation of $\Lambda$. If \eqref{eq:I>0} holds, then,
\begin{itemize}
  \item[(a)] $\Lambda$ has a Bragg peak at each $\chi \in B_1$ of intensity
\[
\widehat{\gamma_{\Lambda}}(\{ \chi\}) \geq (1-s(\chi)) I >0 \,.
\]
  \item[(b)] For each $0< \eps <1$ and each $\chi \in B_\eps$,  $\Lambda$ has a Bragg peak at $\chi$ of intensity
\[
\widehat{\gamma_{\Lambda}}(\{ \chi\}) \geq (1-\eps) I \,.
\]
\end{itemize}
\end{theorem}
\begin{proof}
The proof is similar to the one in Prop.~\ref{prop intensity bragg peaks fib}. Let $(\chi, \chi^\star) \in \cL^0$. Then, for all $t \in L$ we have
\[
1=(\chi, \chi^\star)(t,t^\star)=\chi(t) \chi^\star(t^\star) \,.
\]
Thus, for all $t \in \oplam(W-W)$, we have
\[
\left|1-\chi(t) \right| = \left|1- \overline{\chi^\star(t^\star)} \right| \leq s(\chi) \,.
\]
Now, \cite[Theorem~4.10.14]{MoSt} gives
\begin{align*}
\widehat{\gamma_{\Lambda}}(\{ \chi\}) &= \lim_n \frac{1}{|A_n|} \left( \sum_{z \in (\Lambda -\Lambda)\cap A_n } \overline{\chi(z)} \gamma(\{z \}) \right)  \\
I=\widehat{\gamma_{\Lambda}}(\{ 0\}) &= \lim_n \frac{1}{|A_n|} \left( \sum_{z \in (\Lambda -\Lambda)\cap A_n }  \gamma(\{z \}) \right) \,.
\end{align*}
Using $\supp( \gamma_{\Lambda}) \subseteq  \oplam(W-W)$, we get
\begin{align*}
\left|\widehat{\gamma_{\Lambda}}(\{ 0\}) - \widehat{\gamma_{\Lambda}}(\{ \chi\}) \right|&=\lim_n\frac{1}{|A_n|} \left| \left( \sum_{z \in (\Lambda -\Lambda)\cap A_n } (1-\overline{\chi(z)}) \cdot \gamma(\{z \}) \right) \right| \\
  &\leq \limsup_n\frac{1}{|A_n|} \sum_{z \in (\Lambda -\Lambda)\cap A_n } s(\chi) \cdot \gamma(\{z \}) = s(\chi) \widehat{\gamma_{\Lambda}}(\{ 0\}) \,.
\end{align*}
Therefore, by the triangle inequality and positivity of $\widehat{\gamma}$ we get
\[
 \widehat{\gamma_{\Lambda}}(\{ \chi\}) \geq  \widehat{\gamma_{\Lambda}}(\{0\})- s(\chi) \widehat{\gamma_{\Lambda}}(\{ 0\}) =(1-s(\chi)) I \,.
\]
(a) and (b) follow immediately.
\end{proof}
Next, we show that for each weak model set $\oplam(W)$ and each $\eps>0$, we can construct a relatively dense set in $\widehat{G}$ of
common $\eps$-sup almost periods for the pure point component of diffraction spectra, for all subsets of $\oplam(W)$ (see \cite{NS2,NS11} for sup almost periodicity).
\begin{theorem} Let $\oplam(W)$ be any weak model set,  and let $\eta$ be an autocorrelation of $\oplam(W)$ along some averaging sequence $\cA$. Set $\alpha = \widehat{\eta}(\{0\})$, let $\Lambda \subseteq \oplam(W)$, and let $\gamma$ be an autocorrelation of $\Lambda$ along some subsequence of $\cA$. Then, for all $\chi \in B_{\frac{\eps}{\alpha+1}}$, and all $\psi \in \widehat{G}$, we have
\[
\left| \widehat{\gamma}(\{\psi + \chi\})- \widehat{\gamma}(\{\psi\}) \right| < \eps \,.
\]
\end{theorem}
\begin{remark}
The set $B_{\frac{\eps}{\alpha+1}}$ is relatively dense and depends on $\oplam(W)$ and $\eps$, but it is independent of the choice of $\Lambda$.
\end{remark}
\begin{proof}
Note first that $\gamma \leq \eta$.
By \cite[Theorem~4.10.14]{MoSt} we have
\begin{align*}
\left|\widehat{\gamma}(\{ \psi + \chi\}) - \widehat{\gamma}(\{ \psi\}) \right|&=\lim_n\frac{1}{|A_n|} \left| \left( \sum_{z \in (\Lambda -\Lambda)\cap A_n } (\overline{\psi(z)}\overline{\chi(z)}-\overline{\psi(z)}) \cdot \gamma(\{z \}) \right) \right| \\
  &\leq \limsup_n\frac{1}{|A_n|} \sum_{z \in (\Lambda -\Lambda)\cap A_n } \frac{\eps}{\alpha+1} \cdot \eta(\{z \})\\
  & = \frac{\eps}{\alpha+1} \widehat{\eta}(\{ 0\}) = \frac{\eps}{\alpha+1} \alpha < \eps \,.
\end{align*}
This proves the claim.
\end{proof}
\bigskip
Let us complete the section by observing that the proof of Theorem~\ref{thm bragg meyer set} extends easily to positively weighted Dirac combs with weak Meyer set support. Indeed, the proof of Theorem~\ref{thm bragg meyer set} only uses the fact that $\gamma$ is a positive measure, and that $\supp(\gamma) \subseteq \oplam(W-W)$. Therefore, in a similar way, we get:
\begin{theorem} Let $\mu$ be any positive, translation bounded measure, with the property that $\supp(\mu) \subseteq \oplam(W)$. Let $\gamma$ be any autocorrelation of $\mu$, and assume that \eqref{eq:I>0} holds. Then,
\begin{itemize}
  \item[(a)] The measure $\mu$ has a Bragg peak at each $\chi \in B_1$ of intensity
\[
\widehat{\gamma}(\{ \chi\}) \geq (1-s(\chi)) I >0 \,.
\]
  \item[(b)] For each $0< \eps <1$ and each $\chi \in B_\eps$,  $\mu$ has a Bragg peak $\chi$ of intensity
\[
\widehat{\gamma}(\{ \chi\}) \geq (1-\eps) I \,.
\]\qed
\end{itemize}
\end{theorem}

\section{The ping-pong lemma for Model sets}\label{sect ping}

In this section we review and slightly improve the ping-pong lemma of \cite{NS20a}. This result can be seen as the result for subsets of weak model sets which is similar to the periodicity of diffraction for subsets of lattices (see Theorem~\ref{thm:1}) . We will discuss this further in Section~\ref{sect:diss}.

To make the presentation easier to follow we introduce the following definition.

\begin{definition}\label{def:ping-pong} Let $\Lambda \subseteq G$ be uniformly discrete. We say that $\omega \in \cM(G)$ is a \textbf{ping-pong measure} for $\Lambda$ if it satisfies the following properties:
\begin{itemize}
  \item{} $\supp(\omega)$ is uniformly discrete.
  \item{} $\omega(\{ x \})=1$ for all $x \in \Lambda$.
  \item{} $\omega$ is twice Fourier transformable.
  \item{} $\widehat{\omega}$ is pure point.
\end{itemize}
\end{definition}

Since every twice Fourier transformable measure is translation bounded (see for example \cite[Thm.~4.9.23]{MoSt}), any ping-pong measure is automatically translation bounded.

\smallskip

Let us now recall the following result, which shows that every weak Meyer set admits a ping-pong measure. We should emphasize here that the existence of a ping-pong measure
for each choice of $y \notin W$ is used in the proof of Thm.~\ref{existence geb}, and will be important to establish Theorem~\ref{FCDM} (c), and its consequences.

\begin{proposition}\label{existence ping pong}\cite{NS20a}Let $(G,H, \cL)$ be a CPS, let $W \subseteq H$ be a compact set and let $y \in H \backslash W$.  Then,
\begin{itemize}
  \item[(a)] There exists some $h \in K_2(H)=\Span \{ f*g :f,g \in \Cc(H)\}$ such that $h\equiv 1$ on $W$ and $h(y)=0$.
  \item[(b)] For any $h$ as in (a), the measure
\[
  \omega= \omega_h :=\sum_{(x,x^\star) \in \cL} h(x^\star) \delta_x
\]
  is a ping-pong measure for $\oplam(W)$. Moreover,
\[
\widehat{\omega}=\dens(\cL) \omega^\star_{\check{h}}:= \dens(\cL) \sum_{(\chi,\chi^\star) \in \cL^0} h(\chi^\star) \delta_\chi \,.
\]
\end{itemize}
\end{proposition}
\begin{proof}This follows from \cite[Lemma~3.1 and Proposition~3.2]{NS20a} and their proofs. Note here that $\omega$ is twice Fourier transformable
by \cite[Theorem~4.12]{CRS}.
\end{proof}

\smallskip

Let us note here in passing that most results in this section are based on the fact that each weak Meyer set $\Lambda$
admits a ping-pong measure $\omega$. It follows that many of these results can be generalized to
Fourier transformable measures $\gamma$, such that $\supp(\gamma)$ admits a ping-pong measure.

\smallskip

We can now prove the following slight improvement of the ping-pong lemma \cite[Lemma~3.3]{NS20a}. Note that in \cite[Lemma~3.3 (ii)]{NS20a} we have the extra assumption that $\gamma$ is Fourier transformable, which is actually unnecessary.

First, to make the proof easier to follow, we introduce the following definition. If $\omega$ is a  ping-pong measure for $\oplam(W)$ we say that $\varphi \in \Cc(G)$ is an \textbf{auxiliary function} for $\omega$ if $\varphi*\tilde{\varphi}(0)=1$ and $\supp(\omega)$ is
$\supp(\varphi*\tilde{\varphi})$-uniformly discrete. It is easy to see that any ping-pong measure admits auxiliary functions.

\begin{theorem}[Ping-pong lemma for cut-and-project sets]\label{ping pong theorem}
Let $(G,H, \cL)$ be a CPS, let $W \subseteq H$ be a compact set, let $\omega$ be a ping-pong measure for $\oplam(W)$ and let $\varphi$ be an auxiliary function. Then,
\begin{itemize}
  \item[(Ping)]If $\gamma$ is any Fourier transformable measure with $\supp(\gamma) \subseteq \oplam(W)$, then $\nu:= \left| \widecheck{\varphi} \right|^2 \widehat{\gamma}$ is a finite measure and
    \begin{displaymath}
\widehat{\gamma}=\widehat{\omega}* \nu \,.
    \end{displaymath}
  \item[(Pong)] If $\nu$ is a finite measure on $\widehat{G}$, then $\gamma:= (\widecheck{\nu}) \omega$ satisfies $\supp(\gamma) \subseteq \supp(\omega)$, $\gamma$ is Fourier transformable, and
    \begin{displaymath}
\widehat{\gamma}=\widehat{\omega}* \nu \,.
    \end{displaymath}
\end{itemize}
\end{theorem}
\begin{proof}
 \textbf{(Ping)} Follows from \cite[Lemma~3.3 (i)]{NS20a}.

\textbf{(Pong)} The only thing which was not proved in \cite[Lemma~3.3 (ii)]{NS20a} is the Fourier transformability of $\gamma$. We prove it here.

First, as we noted after Def.~\ref{def:ping-pong}, since $\omega$ is twice Fourier transformable, it is translation bounded. By \cite[Theorem~8.5]{NS20a}
the ping-pong measure $\omega$ is a linear combination of positive definite measures. The finite measure $\nu$ is a linear combination
of finite positive measures, and hence $\widehat{\nu}$ is a linear combination of continuous positive definite functions. Since the product of a continuous positive definite function, and a positive definite measure, is a positive definite measure \cite[Corollary~4.3]{ARMA1}, it follows that the measure
\[
\gamma = (\widecheck{\nu}) \omega
\]
is a linear combination of positive definite measures. Therefore, $\gamma$ is Fourier transformable. Moreover, $\widecheck{\nu} \in \Cu(G)$ and $\omega \in \cM^\infty(G)$ imply $\gamma \in \cM^\infty(G)$. As $\gamma$ is supported inside a Meyer set, by \cite[Theorem~5.7]{CRS} (or \cite[Theorem~4.9.32]{MoSt}), $\gamma$ is twice Fourier transformable.

Next, by \cite[Theorem 4.9.28]{MoSt} or \cite[Theorem~3.4]{ARMA} we have
\[
\widehat{ \hat{\gamma}} = \gamma^\dagger \,.
\]
Finally, $\omega $ is twice Fourier transformable by definition. Thus, by \cite[Lemma~4.9.26]{MoSt}, the measure $\widehat{\omega}*\nu$ is Fourier transformable, and
\[
\reallywidehat{\widehat{\omega}*\nu}= \widehat{\widehat{\omega}} \widehat{\nu}= \omega^\dagger \widehat{\nu}= \omega^\dagger \check{\nu}^\dagger = \gamma^\dagger \,.
\]
This shows that
\[
\reallywidehat{\widehat{\omega}*\nu}= \widehat{ \hat{\gamma}} \,.
\]
The injectivity of the Fourier transform \cite[Theorem~4.9.13]{MoSt} then gives $\widehat{\gamma}=\widehat{\omega}*\nu$. This completes the proof.
\end{proof}

\medskip

Next, let us recall the  \textbf{Fourier--Stieltjes algebra $B(G)$ of $G$},
\[
B(G):=\{ \widehat{\mu} : \mu \mbox{ is a finite measure on } \widehat{G} \} \,.
\]
The following is an immediate consequence of the ping-pong lemma:

\begin{theorem}\label{thm:BG} Let $(G,H,\cL)$ be a CPS, $W \subseteq H$ be a compact set, and let $\omega$ be a ping-pong measure for $\oplam(W)$.
Then, for any measure $\gamma$ with $\supp(\gamma) \subseteq \oplam(W)$, the following are equivalent:
\begin{itemize}
  \item[(i)] $\gamma$ is Fourier transformable.
  \item[(ii)] $\gamma$ is a linear combination of positive definite measures.
  \item[(iii)] There exists some $f \in B(G)$ such that
\begin{equation}\label{eq:BG}
  \gamma = f \omega \,.
\end{equation}
\end{itemize}
\end{theorem}
\begin{proof}
\noindent
(i) $\Longrightarrow$ (ii) follows from \cite[Theorem~8.5]{NS20a}.

\noindent
(ii) $\Longrightarrow$ (i) is obvious.

\noindent
(i) $\Longrightarrow$ (iii) By the ping-pong lemma $\widehat{\gamma}=\omega*\nu$ for some finite measure $\nu$, and
\[
\gamma = \check{\nu} \omega \,.
\]
Then $f= \check{\nu} \in B(G)$ satisfies \eqref{eq:BG}.

\noindent
(iii) $\Longrightarrow$ (i) Let $f \in B(G)$ be so that \eqref{eq:BG} holds. Let $\nu$ be the finite measure such that $f =\check{\nu}$. Then, by the Pong implication in the ping-pong lemma, the measure $\gamma= f \omega$ is Fourier transformable and
\[
\widehat{\gamma}=\widehat{\omega}*\nu \,.
\]
\end{proof}

For subsets of lattices, a similar result is proved in \cite[Thm.~1]{Ba}.
\medskip

Combining Theorem~\ref{thm:BG} with \cite[Theorem~4.1]{NS20a} we get:

\begin{theorem}[Existence of the generalized Eberlein decomposition]\cite[Theorem~4.1]{NS20a}\label{existence geb} Let $(G,H,\cL)$ be a CPS, and let $W \subseteq H$ be any compact set.

Let $\gamma$ be any Fourier transformable measure  with $\supp(\gamma) \subseteq \oplam(W)$. Then, there exist unique Fourier transformable measures $\gamma_{\mathsf{s}}, \gamma_{0\mathsf{s}}, \gamma_{0\mathsf{a}}$, supported inside $\oplam(W)$, such that
\begin{align*}
\gamma=\gamma_{\mathsf{s}} &+ \gamma_{0\mathsf{s}} + \gamma_{0\mathsf{a}} \\
\reallywidehat{\gamma_{\mathsf{s}}}= \left(\widehat{\gamma}\right)_{\mathsf{pp}} \,;\quad
\reallywidehat{\gamma_{0\mathsf{s}}} &=  \left(\widehat{\gamma}\right)_{\mathsf{sc}} \,;\quad
\reallywidehat{\gamma_{0\mathsf{a}}} =\left(\widehat{\gamma}\right)_{\mathsf{ac}} \,.  \\
\end{align*}
Moreover, given some ping-pong measures $\omega$ for $\oplam(W)$, there exist some functions $f_{\mathsf{s}}, f_{0\mathsf{s}}, f_{0\mathsf{a}} \in B(G)$, which are Fourier transforms of finite pure point, singular continuous, and absolutely continuous measures, respectively, such that
\begin{displaymath}
\gamma_{\mathsf{s}}= f_{\mathsf{s}} \omega \,;\,  \gamma_{0\mathsf{s}} =  f_{0\mathsf{s}} \omega \,;\, \gamma_{0\mathsf{a}} =f_{0\mathsf{a}} \omega \,.
\end{displaymath} \qed
\end{theorem}

We will show that a similar decomposition holds, under much more general settings, in Section~\ref{sect SCF}.

\medskip

\section{Norm estimates for the diffraction of measures supported inside model sets}

If $\omega$ is a ping-pong measure for some weak model set $\oplam(W)$, then $\widehat{\omega}$ is a strong almost periodic measure, but it is not
norm almost periodic in general (see \cite{MoSt,NS11} for definition and properties of strong and norm almost periodic measures). If we can find a single ping-pong measure $\omega$ for $\oplam(W)$, such that $\widehat{\omega}$ is norm almost periodic, then it follows easily that, for each transformable measure $\gamma$ supported within $\oplam(W)$, the measure $\widehat{\gamma}$ is norm almost periodic. The existence of such a measure is shown in the proof of \cite[Theorem~7.1]{NS20a}:

\begin{proposition}\cite[Theorem~7.1]{NS20a}\label{prop nap} Let $(G,H,\cL)$ be a CPS, and let $W \subseteq H$ be any compact set. Then, there exists a ping-pong measure $\omega$ for $\oplam(W)$ such that $\widehat{\omega}$ is norm almost periodic. \qed
\end{proposition}

\begin{remark} In Prop.~\ref{prop nap}, a norm almost periodic ping-pong measure can be constructed explicitly as follows:

Pick any compact set $W' \supseteq W$ with non-empty interior, and define $H':= \langle H' \rangle$ to be the subgroup of $H$ generated by $W'$. Then, $H'$ is compactly generated, and an open and closed subgroup of $H$. It follows that $H/H'$ is discrete.

Then, $(G,H', \cL')$ is a CPS, where
\[
\cL':= \{ (x,x^\star): x \in L, x^\star \in H' \} \,.
\]
Next, the structure theorem of compactly generated LCAG \cite{ReiSte}, implies that there exists some $d, n \geq 0$ and a compact group $\KK$ such that $H' \simeq \RR^d \times \ZZ^m \times \KK$. Denote by $\psi :H' \to \RR^d \times \ZZ^m \times \KK$ such an isomorphism. Then, by compactness, there exists some $R>0, N>0$ such that
\[
\psi(W) \subseteq [-R,R]^d \times [-N,N]^m \times \KK \,.
\]
Next, pick some $\varphi \in \Cc^\infty(\RR^d)\cap K_2(\RR^d)$, such that $\varphi(x)=1 \,,\, \forall x \in [-R,R]^d$, and define:
\[
\omega:= \sum_{(x,x^\star) \in \cL'} (\varphi \otimes 1_{[-N,N]^m} \otimes 1_{\KK})( \psi (x^\star)) \delta_x \,.
\]
Then, $\omega$ is a ping-pong measure for $\oplam(W)$, and $\widehat{\omega}$ is norm almost periodic \cite{NS20a}.
\end{remark}

Let us now give the following estimate for the norm almost periods of the diffraction, for any measure supported inside $\oplam(W)$.

\begin{proposition}\label{P1} Let $\oplam(Q)$ be any weak model set, and let $\omega$ be any ping-pong measure for $\oplam(W)$. Let $\eta$ be any positive definite measure supported inside $\oplam(W)$. Let $K \subseteq \mbox{\rm Int} (K_1)$ be any two compact sets in $\widehat{G}$. Then, for all $\chi \in \widehat{G}$, we have
\begin{displaymath}
\| T_\chi \widehat{\gamma}-\widehat{\gamma} \|_{K} \leq \| T_\chi \widehat{\omega} -\widehat{\omega} \|_{K_1} \gamma( \{ 0\}) \,.
\end{displaymath}
\end{proposition}
\begin{proof}
Let $\varphi$ be an auxiliary function, and as usual, set
\[
\nu= \left| \widecheck{\varphi} \right|^2 \, \widehat{\gamma} \,.
\]
Then, by the ping-pong lemma, we have
\begin{displaymath}
\| T_\chi \widehat{\gamma}-\widehat{\gamma} \|_{K}= \| \left(T_\chi \widehat{\omega}-\widehat{\omega}\right)* \nu \|_{K} \leq \| \left(T_\chi \widehat{\omega}-\widehat{\omega}\right) \|_{K_1} |\nu|(\widehat{G}) \,.
\end{displaymath}
Now, since $\gamma$ is positive definite, the measure $\nu$ is positive, and therefore we have
\[
|\nu|(\widehat{G})=\nu(\widehat{G})=\widehat{\gamma}(\left| \widecheck{\varphi} \right|^2)=\gamma(\varphi*\tilde{\varphi})= \gamma(\{0\}) \,,
\]
with the last equality following from the fact that $\varphi$ is an auxiliary function. This proves the claim.
\end{proof}

\medskip
\begin{tcolorbox}
For the remainder of this section, $(G,H, \cL)$ is a fixed CPS, $W \subseteq H$ and $K \subseteq \widehat{G}$ are fixed compact sets, and $\omega$ is a fixed ping-pong measure, such that, $\widehat{\omega}$ is norm almost periodic.

We fix another compact set $K_1 \subseteq \widehat{G}$, such that $K \subseteq \mbox{\rm Int}(K_1)$. For each $\eps >0$ we set
\[
P_\eps := \left\{ \chi \in \widehat{G} : \| T_\chi \widehat{\omega} - \widehat{\omega} \|_{K_1} < \eps \right\} \,.
\]
\end{tcolorbox}
Note that $P_\eps$ is relatively dense for all $\eps >0$, and only depends on the ping-pong measure $\omega$ and the choice of $K_1$.
\medskip

Combining all the results of this section, we get:

\begin{theorem}\label{NAP} Let
\[
\mu = \sum_{x \in \oplam(W)} c_x \delta_x
\]
be translation bounded, with autocorrelation $\gamma$, and set $c:= \sup \{ |c_x| :x \in \oplam(W) \} $. Then,
for all $\chi \in P_\eps$ and all $\alpha \in \{\mathsf{pp}, \mathsf{ac}, \mathsf{sc} \}$  we have
\[
\|T_\chi \widehat{\gamma}-\widehat{\gamma}\|_K \,,\,   \|T_\chi (\widehat{\gamma})_{\alpha}-(\widehat{\gamma})_{\alpha}\|_K \leq \eps \left(\gamma(\{0\})\right) \leq c^2 \eps \dens(\cL^0)  \theta_H(W) \,,
\]
where $\theta_H$ denotes the Haar measure on $H$.
\end{theorem}
\begin{proof}
Since $\supp(\gamma) \subseteq \oplam(W-W)$, by Theorem~\ref{existence geb} we have
\[
\supp((\gamma)_{\mathsf{s}}) \,,\, \supp((\gamma)_{0\mathsf{a}}) \,,\,\supp((\gamma)_{0\mathsf{s}}) \subseteq \oplam(W-W) \,.
\]
Therefore, by Proposition~\ref{P1} we get
\[
  \|T_\chi \widehat{\gamma}-\widehat{\gamma}\|_K \,,\,  \|T_\chi (\widehat{\gamma})_{\alpha}-(\widehat{\gamma})_{\alpha}\|_K\leq \eps \gamma(\{0\}) \,.
\]
Now, positive definiteness gives
\begin{displaymath}
0 \leq  \gamma_{\mathsf{s}}(\{0\}), \gamma_{0\mathsf{a}}(\{0\}), \gamma_{0\mathsf{s}}(\{0\}) \leq \gamma_{\mathsf{s}}(\{0\}) + \gamma_{0\mathsf{a}}(\{0\}) + \gamma_{0\mathsf{s}}(\{0\}) = \gamma(\{0\}) \,.
\end{displaymath}
To complete the proof, we observe that
\begin{align*}
  \gamma(\{0\}) & =\lim_n \frac{\sum_{x \in (\Lambda \cap A_n)} c_x \overline{c_x} }{|A_n|} =\lim_n \frac{\sum_{x \in (\Lambda \cap A_n)} |c_x|^2 }{|A_n|} \\
   &\leq c^2 \udens_{\cA}(\oplam(W))  \leq c^2 \dens(\cL^0) \theta_H(W) \,.
\end{align*}
\end{proof}

As an immediate consequence, we get a common relatively dense set of $\eps$-almost periods for the diffraction of all subsets $\Lambda \subseteq \oplam(W)$:

\begin{theorem}\label{NAPMS} Let $\Lambda \subseteq \oplam(W)$, let $\eps>0$, and let $\gamma$ be an autocorrelation of $\Lambda$. Set $\eps':=\frac{\eps}{\dens(\cL^0)\theta_H(W)+1}$. Then,
for all $\chi \in P_{\eps'}$ we have
\[
\max \{
  \|T_\chi \widehat{\gamma}-\widehat{\gamma}\|_K \,,\,
  \|T_\chi (\widehat{\gamma})_{\mathsf{pp}}-(\widehat{\gamma})_{\mathsf{pp}}\|_K \,,\,
  \|T_\chi (\widehat{\gamma})_{\mathsf{ac}}-(\widehat{\gamma})_{\mathsf{ac}}\|_K \,,\,
  \|T_\chi (\widehat{\gamma})_{\mathsf{sc}}-(\widehat{\gamma})_{\mathsf{sc}}\|_K \} \leq \eps \,.
\]\qed
\end{theorem}
Similarly, given any gamily of equi-translation bounded measures supported inside $\oplam(W)$, we can construct a common relatively dense set of $\eps$-almost periods for the diffraction of all the measures in the family.

\begin{theorem}\label{NAPMSwc} For each $\eps>0$ and each $\alpha >0$, there exists a relatively dense set $Q=Q(\eps,\alpha,\oplam(W)) \subseteq \widehat{G}$ with the following property :

\noindent Let $\mu = \sum_{x \in \oplam(W)} c_x \delta_x$ be a measure with $\sup \{ |c_x| \}  \leq \alpha$, and let $\gamma$ be an autocorrelation of $\mu$. Then, for all $\chi \in Q$ we have
\[
\max \{
  \|T_\chi \widehat{\gamma}-\widehat{\gamma}\|_K \,,\,
  \|T_\chi (\widehat{\gamma})_{\mathsf{pp}}-(\widehat{\gamma})_{\mathsf{pp}}\|_K \,,\,
  \|T_\chi (\widehat{\gamma})_{\mathsf{ac}}-(\widehat{\gamma})_{\mathsf{ac}}\|_K \,,\,
  \|T_\chi (\widehat{\gamma})_{\mathsf{sc}}-(\widehat{\gamma})_{\mathsf{sc}}\|_K \} \leq \eps \,.
\]

\end{theorem}
\begin{proof}
Let $\eps ':= \frac{\eps}{\alpha^2 \dens(\cL^0) \theta_H(W)+1}$, and set $Q:= P_{\eps'}$. The claim follows from Theorem~\ref{NAP}.
\end{proof}
As an immediate consequence, we get:

\begin{corollary}\cite[Theorem~7.1 and Corollary~7.2]{NS20a} Let $\mu$ be any measure with weak Meyer set support and let $\gamma$ be an autocorrelation of $\mu$. Then $\widehat{\gamma}, (\widehat{\gamma})_{\mathsf{\mathsf{pp}}},(\widehat{\gamma})_{\mathsf{\mathsf{ac}}}$ and $(\widehat{\gamma})_{\mathsf{\mathsf{sc}}}$ are norm almost periodic.\qed
\end{corollary}


\section{Diffraction of weak Meyer sets and subsets of lattices}\label{sect:diss}

Let us briefly compare the diffraction properties of subsets of model sets, to those for subsets of lattices. Let us recall first the following result, which is a particular case of \cite[Theorem~1]{Ba}.

\begin{theorem}\label{thm:1}\cite{Ba} Let $L$ be a lattice in $\RR^d$ and let $S \subseteq L$ be arbitrary. Let $\gamma$ be any autocorrelation of $S$. Then, there exists a finite positive measure $\nu$ such that
  \[
  \widehat{\gamma}=\delta_{L^0}*\nu \,.
  \]\qed
\end{theorem}

As mentioned above, the (Ping) part of ping-pong lemma is the result for subsets of model sets similar to Theorem~\ref{thm:1}.

\smallskip

Let us note here that Theorem~\ref{thm:1} implies that the set of Bragg peaks of $S$ of the highest intensity $I=\widehat{\gamma}(\{0\})$ contains the dual lattice $L^0$. Similarly, Theorem~\ref{thm bragg meyer set} shows that, for subsets $\Lambda$ of model sets, the set of Bragg peaks of $\Lambda$ of the intensity almost equal to $I=\widehat{\gamma}(\{0\})$ contains the dual model set $B_\eps=\hoplam(W_\eps)$.

Furthermore, Theorem~\ref{thm:1} gives that the absolutely continuous, and singular continuous components of diffraction spectrum of $S$, are $L^0$-periodic. Theorem~\ref{NAPMSwc} gives that the absolutely continuous and singular continuous components of diffraction spectrum of weak Meyer sets are almost periodic, and that the set of almost periods contain model sets in the dual CPS, with small neighbourhoods of zero as windows.

The similarity, and relationship, between the diffraction of subsets of lattices and subsets of model sets becomes more clear in the case of CPS with compact internal space $H$. Indeed, in this case, $L= \oplam(H)$ is a lattice containing our weak Meyer set $\Lambda$, and $\hoplam(W_\eps)$ contains the model set $\hoplam(\{ 0\})=L^0$.

\section{FCDM functions}\label{sect SCF}

In this section, we look at potential decompositions of the diffraction of Meyer sets, which are similar to the existence of the (generalized) Eberlein decomposition. Reading carefully the proof of Thm.~\ref{existence geb}, one can see that it only relies on the fact that the Lebesgue decomposition satisfies two simple properties, which we list in the following definition.

\begin{definition}\label{def FCDM} A function $P : \cM^\infty(G) \to \cM^\infty(G)$ is said to be \textbf{a function which factors through convolution with discrete measures}, or simply a \textbf{FCDM} function, if it satisfies the following two conditions:
\begin{itemize}
  \item{} If $\nu$ is finite, then $P(\nu)$ is finite.
  \item{} For all pure point measures $\omega \in \cM^\infty(G)$, and all finite measures $\nu$, we have
  \[
P(\omega*\nu) = \omega* P(\nu) \,.
  \]
\end{itemize}
\end{definition}

Note here that, when $\omega$ is translation bounded, and $\nu$ is finite, both convolutions $\omega*\nu$ and $\omega* P(\nu)$ are well defined, and translation bounded \cite{ARMA1,MoSt}.

\begin{example} The projections $P_{\mathsf{pp}}, P_{\mathsf{ac}}, P_{\mathsf{sc}}$ on the pure point, absolutely continuous and singular continuous components, respectively, are FCDM functions.
\end{example}

\smallskip

The following lemma is obvious.
\begin{lemma}
\begin{itemize}
  \item[(a)] The identity map $\mbox{Id} : \cM^\infty(G) \to \cM^\infty(G)$ is a FCDM function.
  \item[(b)] If $P,Q : \cM^\infty(G) \to \cM^\infty(G)$ are FCDM function, and $C_1,C_2 \in \CC$, then $C_1P+C_2Q$ is a  FCDM function.
\end{itemize}\qed
\end{lemma}

This means that every FCDM function $P$, gives a ``canonical" decomposition for each $\mu \in \cM^\infty(G)$, as
\[
\mu= P(\mu)+Q(\mu) \,,
\]
where $Q=\mbox{Id}-P$ is also a FCDM function.
\smallskip

\begin{remark} Let $P : \cM^\infty(G) \to \cM^\infty(G)$ be any FCDM function. Then,
\begin{itemize}
  \item[(a)] For each finite measure $\mu$ and $C \in \CC$, we have
\[
P(C \mu) = P((C\delta_0)*\mu)=(C \delta_0)*P(\mu)=CP(\mu) \,.
\]
  \item[(b)] For each finite measure $\mu$ and $t \in G$, we have
\[
P(T_t \mu)=P(\delta_t*\mu)= \delta_t*P(\mu)=T_t P(\mu) \,.
\]
In particular, if $\tau$ is a topology on $\cM^\infty(G)$, with the property that the set of finite measures is dense in $\cM^\infty(G)$, and $P$ is continuous with respect to $\tau$, then $P$ commutes with the translation operators $T_t$.

\item[(c)] If $P : \cM^\infty(G) \to \cM^\infty(G)$ is linear, and commutes with the translate operators $T_t$, then, for all finite pure point measure $\omega$, and for all finite measures $\mu$, we have $P(\omega*\mu)=\omega*P(\mu)$.

    In particular, if $\tau$ is any topology on $\cM^\infty(G)$ with the property that, for all pure point measures $\omega$, and all finite measures $\mu$, we have
\[
\omega*\mu =\lim_{\substack{F \subseteq G \\ F \mbox{ \rm is finite}}} (\sum_{x \in F} \omega(\{x\})\delta_x) * \mu \,,
\]
and $P : (\cM^\infty(G) , \tau) \to (\cM^\infty(G), \tau)$ is a continuous linear function, then $P$ is a FCDM function if and only if $P$ commutes with the translation operators.
\end{itemize}
\end{remark}

We can now prove the following result, which is a generalisation of the existence of the general Eberlein decomposition for measures with Meyer set support.

\begin{theorem}\label{FCDM} Let $(G,H,\cL)$ be a CPS, let $W \subseteq H$ be any compact set, and let $\omega$ be a ping-pong measure for $\oplam(W)$. Let $P: \cM^\infty(\widehat{G}) \to \cM^\infty(\widehat{G})$ be any FCDM function.

Then, for each Fourier transformable measure $\gamma$ with $\supp(\gamma) \subseteq \oplam(W)$, there exists a (unique) Fourier transformable measure $\gamma_{P}$, and some function $f_P \in B(G)$, such that
\begin{itemize}
  \item[(a)] $\displaystyle \widehat{\gamma_{P}}=P( \widehat{\gamma})$.
  \item[(b)]$ \displaystyle \gamma_{P} = (f_P) \omega$.
  \item[(c)] $\supp(\gamma_{P}) \subseteq \oplam(W)$.
\end{itemize}
\end{theorem}
\begin{proof}
By the ping-pong lemma (Theorem~\ref{ping pong theorem}), there exists a finite measure $\nu$ such that
\begin{displaymath}
\widehat{\gamma}= \widehat{\omega}* \nu \,.
\end{displaymath}
Since $\nu$ is a finite measure, it follows from the definition of FCDM function that $P(\nu)$ is finite. Since $\widehat{\omega}$ is translation bounded \cite{MoSt}, the measures $\widehat{\omega}$ and $P(\nu)$ are convolvable, and, by the definition of FCDM we have
\begin{displaymath}
\widehat{\omega}* P(\nu) =P(\widehat{\omega}*\nu)= P(\widehat{\gamma})  \,.
\end{displaymath}
Finally, as $P(\nu)$ is a finite measure, by Theorem~\ref{ping pong theorem}, the measure
\[
\gamma_{P}:= \widecheck{P(\nu)} \omega
\]
is supported inside $\supp(\omega)$, Fourier transformable, and
\begin{displaymath}
\widehat{\gamma_P}= \widehat{\omega}* P(\nu)= P(\widehat{\gamma}) \,.
\end{displaymath}
Setting $f_P:= \widecheck{P(\nu)}$ we get (a) and (b).

To show (c), pick $z \in \pi^G(\cL) \backslash \oplam(W)$. By Proposition~\ref{P1} (a), there exists some $h_1 \in K_2(G)$ such that $\omega_1=\omega_{h_1}$ is a ping-pong measure for $\oplam_1(W)$ and $\omega(\{z \})=0$. Repeating the proof above, with $\omega$ replaced by $\omega_1$, we get that $\gamma_P(\{ z \})=0$. This proves (c).
\end{proof}

\begin{remark} Let $(G,H,\cL)$ be a CPS, and $W \subseteq H$ be any compact set. Denote by
\[
{\mathcal FT}_{W}(G):= \{ \gamma: \gamma \mbox{ is Fourier transformable, and } \supp(\gamma) \subseteq \oplam(W) \} \,.
\]
Then, Theorem~\ref{FCDM} says that every FCDM function $P: \cM^\infty(\widehat{G}) \to \cM^\infty(\widehat{G})$ induces a unique function
\[
R: {\mathcal FT}_{W}(G) \to {\mathcal FT}_{W}(G)
\]
such that, for all $\gamma \in {\mathcal FT}_{W}(G)$ we have
\[
\widehat{R(\gamma)}=P(\widehat{\gamma}) \,.
\]
\end{remark}

Combining these results with \cite[Theorem~7.1]{NS20a} we get:

\begin{corollary} Let $\gamma$ be a Fourier transformable measure supported inside a Meyer set, and let $P: \cM^\infty(\widehat{G}) \to \cM^\infty(\widehat{G})$ FCDM function. Then, $P( \widehat{\gamma})$  is norm almost periodic. In particular, either $P( \widehat{\gamma})=0$, or $P( \widehat{\gamma})$ has relatively dense support.
\end{corollary}

\medskip

\subsection{An application: Hausdorff dimension decomposition}

In this section we will denote by $\dim_{H}$ the Hausdorff dimension on $\RR$. We will often employ the well known formula
\begin{equation}\label{eq:dh}
\dim_{H} (\bigcup_{n \in \NN } S_n) = \sup_{n \in \NN} \dim_{H}(S_n) \,,
\end{equation}
for any countable family $ \{ S_n \}$ of subsets of $\RR$.

Let us recall and extend the following definitions of \cite{Last}.

\begin{definition} Let $\mu$ be a positive measure on $\RR$.
\begin{itemize}
  \item[(a)] $\mu$ is called \textbf{zero-dimensional} if it is supported on a Borel set $S$ of Hausdorff dimension $\dim_{H}(S)=0$.
  \item[(b)] $\mu$ is called \textbf{positive-dimensional} if $\mu(S)=0$ for all Borel sets $S$ with $\dim_{H}(S)=0$.
\end{itemize}
A measure $\nu$ on $\RR$ is called \textbf{zero-dimensional}, or \textbf{positive-dimensional}, respectively, if its total variation measure $|\mu|$ is zero-dimensional, or positive-dimensional, respectively.
\end{definition}

The following result follows immediately from the definitions.

\begin{lemma}\label{lem2} Let $\mu$ be a measure, and consider the canonical decomposition
\[
\mu=\mbox{Re}(\mu)+i \mbox{Im}(\mu)= [(\mbox{Re}(\mu))_+- (\mbox{Re}(\mu))_-]+i[(\mbox{Im}(\mu))_+- (\mbox{Im}(\mu))_-] \,,
\]
given by the Hahn decomposition of $\mbox{Re}(\mu)$ and $\mbox{Im}(\mu)$. Then, $\mu$ is zero-dimensional or positive-dimensional, respectively, if and only if  each of
the four components is zero-dimensional, or positive-dimensional, respectively.
\end{lemma}
\begin{proof}
The proof is an immediate consequence of the fact that $\supp(\mu)$ is the union of the supports of the four components, and that
and
\begin{align*}
&(\mbox{Re}(\mu))_+ ; (\mbox{Re}(\mu))_- ; (\mbox{Im}(\mu))_+ ;((\mbox{Im}(\mu))_-) \leq |\mu| \\
&\leq \mbox{Re}(\mu))_+ + (\mbox{Re}(\mu))_- + (\mbox{Im}(\mu))_+ + (\mbox{Im}(\mu))_- \,.
\end{align*}
\end{proof}

By combining the definition with \eqref{eq:dh}, we immediately get:
\begin{fact}\label{fact2} Let $\mu$ be a positive measure on $\RR$. Then, $\mu$ is zero-dimensional, or positive-dimensional, respectively, if and only if, for all $n$, the restriction $\mu|_{[n, n+1)}$, is  zero-dimensional, or positive-dimensional, respectively.
\end{fact}

This fact, combined with \cite[Cor.~4.1.4(a)]{Last} gives:
\begin{proposition} Any measure $\mu$ on $\RR$ has a unique decomposition
\[
\mu=\mu_{zd}+\mu_{pd} \,,
\]
where $\mu_{zd}$ is zero-dimensional, and $\mu_{pd}$ is positive-dimensional. Moreover, the mappings $P_{pd}(\mu):= \mu_{pd}$, and $P_{zd}(\mu):= \mu_{zd}$ are positive.
\end{proposition}
\begin{proof}
\cite[Cor.~4.1.4(a)]{Last} gives the results for finite positive measures. Fact~\ref{fact2} then extends the result to positive measures, while Lemma~\ref{lem2} proves the general version.
\end{proof}

We now have the following simple result.
\begin{lemma} $P_{pd}$ and $P_{zd}$ are FCDM functions.
\end{lemma}
\begin{proof}
It is obvious that these functions take finite measures to finite measures. Next, let $\omega= \sum_{n=1}^\infty c_n \delta_{t_n}$ be a pure point measure, and let $\nu$ be a finite measure.

If $\nu$ is zero-dimensional, then $\omega*\nu$ is supported inside $\bigcup_{n} (t_n+\supp(\nu))$, which is zero-dimensional by \eqref{eq:dh}.

Similarly, if $\nu$ is positive-dimensional, then for all Borel sets $S$ with $\dim_{H}(S)=0$, we have
\[
|\omega*\nu| \leq \sum_{n=1}^\infty |c_n| |\nu|(-t_n+S) =0 \,,
\]
and hence $\omega*\nu$ is positive-dimensional. Therefore,
\[
P_{pd}(\omega*\nu)+ P_{zd}(\omega*\nu)=\omega*\nu = \omega*(P_{pd} \nu)+ \omega*(P_{zd} \nu)
\]
give two decompositions of $\omega*\nu$, into a positive-dimensional and a zero-dimensional measure, respectively. The uniqueness of the decomposition proves the claim.
\end{proof}

Now, Theorem~\ref{FCDM} gives:

\begin{corollary} Let $(G,H,\cL)$ be a CPS, let $W \subseteq H$ be any compact set, and let $\omega$ be a ping-pong measure for $\oplam(W)$.
Then, for each Fourier transformable measure $\gamma$, with $\supp(\gamma) \subseteq \oplam(W)$, there exists (unique) Fourier transformable measures $\gamma_{1},\gamma_{2}$, and some function $f_{1},f_{2}\in B(G)$, such that:
\begin{itemize}
  \item[(a)] $\displaystyle \widehat{\gamma_{1}}=P_{pd}( \widehat{\gamma})$ and $\widehat{\gamma_{2}}=P_{zd}( \widehat{\gamma}) \,.$
  \item[(b)]$ \displaystyle \gamma_{1} = f_1 \omega$ and $ \displaystyle \gamma_{2} = f_2 \omega$.
  \item[(c)] $\supp(\gamma_{1}), \supp(\gamma_{2}) \subseteq \oplam(W)$.\qed
\end{itemize}
\end{corollary}

It is likely that similar results can be proved for most decompositions of \cite[Section~4]{Last}.

\section{Outlook}

The study of Meyer sets via covering model sets is not a new idea in the field of aperiodic order. This idea already plays a central role in the characterisation of Meyer sets (see \cite{Meyer,MOO,NS11} for example) and comes up in a natural way in the Pisot conjecture. Any 1-dimensional primitive Pisot substitution gives raise to a cut-and-project scheme and to an iterated function system (IFS) on the internal space. Using the attractors of the IFS as windows, we get covering model sets for the left-end points of each tile type in the fixed point of the substitution (see \cite[Chapter 7]{TAO} or \cite{BG3} for some examples and \cite{Sin} for a systematic exposition of 1-dimensional Pisot substitutions). The covering model sets are regular and hence have pure point diffraction spectra. If the inflation factor is an irrational Pisot number, the Pisot conjecture is equivalent to the fact that the fixed point of the substitution and the covering model set differ on a set of density zero. On another hand, when the inflation factor is an integer it is possible for the fixed point to have mixed spectrum, compare \cite{BG3}.

\smallskip

There are some limitations when studying the diffraction of a weak Meyer sets via covering model sets. The results of the paper can be understood the following equivalent way: starting with a (regular) model set $\oplam(W)$ with compact window $W$, we can describe
common properties for all diffraction measures of all subsets of $\oplam(W)$. The set $\oplam(W)$ contains many subsets, some which have pure point spectrum, some with mixed spectrum, but also many subsets $\Lambda$ which have multiple different diffraction measures of distinct spectral type. The last situation occurs often in the case of
maximal density weak model sets with windows of empty interior,such as the square free integers and visible points of the lattice. It would be interesting to classify all the subsets of a given model set and all the averaging sequences which lead to a diffraction measure of a certain spectral type, but this seems to be an extremely difficult question.

\subsection*{Acknowledgments}  The work was supported by NSERC with grant 2020-00038, and we are greatly thankful for all the support. The generalisation to FCDM functions, and its application were suggested by Daniel Lenz, and we are grateful for his insightful feedback. We would like to thank Anna Klick for creating Figure~\ref{fib1} in Maple, and for some suggestions which improved the quality of the manuscript. We would like to thank Michael Baake for many helpful suggestions, and Adam Humeniuk for carefully reading the manuscript and for many suggestions which helped improve the paper. Finally, we would like to thank the two anonymous referees for numerous suggestions which improved the quality of this manuscript.


\begin{thebibliography}{99}


\bibitem{ARMA1} L.~N ~Argabright, J. ~Gil ~de ~Lamadrid, \textit{ Fourier Analysis of
Unbounded Measures on Locally Compact Abelian Groups}, Memoirs Amer. Math. Soc., Vol \textbf{145}, 1974.

\bibitem{Ba}
M.~Baake, \textit{Diffraction of weighted lattice subsets}, Canad. Math. Bull. \textbf{45}, 483--498, 2002.
\texttt{arXiv:1511.00885}.

\bibitem{BaGa}
M.~Baake, F.~G\"{a}hler, \textit{Pair correlations of aperiodic inflation rules via
renormalisation: Some interesting examples}, Topol. Appl. \textbf{205}, 4--27, 2016.
\texttt{arXiv:1511.00885}.

\bibitem{TAO}
M.~Baake, U.~Grimm, \textit{Aperiodic Order. Vol.~1: A Mathematical Invitation}, Cambridge University Press, Cambridge, 2013.

\bibitem{TAO2}
M.~Baake, U.~Grimm (eds.), \textit{Aperiodic Order. Vol.~2:  Crystallography and Almost Periodicity}, Cambridge University Press, Cambridge, 2017.


\bibitem{BG2} M. Baake, U. Grimm,  \textit{Squirals and beyond: Substitution tilings with singular continuous spectrum},
Ergod. Th. \& Dynam. Sys. \textbf{34}, 1077--1102, 2014.
\texttt{arXiv:1205.1384}

\bibitem{BG3} M. Baake, U. Grimm,  \textit{Fourier transform of Rauzy fractals and point spectrum of 1D Pisot inflation tilings},
Documenta Mathematica \textbf{25}, 2303--2337, 2020. 
\texttt{arXiv:1907.11012}


\bibitem{BHS}
M.~Baake, C.~Huck, N.~Strungaru, \textit{On weak model sets of extremal density}, Indag.~Math. \textbf{28}, 3--31, 2017.
\texttt{arXiv:1512.07129v2}

\bibitem{BL}
M.~Baake, D.~Lenz, \textit{Dynamical systems on translation bounded
measures:\ Pure point dynamical and diffraction spectra}, Ergod.\ Th.\ \& Dynam.\ Sys. \textbf{24}, 1867--1893, 2004.
\texttt{arXiv:math.DS/0302231}.

\bibitem{BLM}
M. ~Baake, D. Lenz, R.V. Moody, \textit{A characterisation of model
sets via dynamical systems}, Ergod. Th. \& Dynam. Sys. \textbf{27},
341--382, 2007. \texttt{arXiv:math.DS/0511648}.


\bibitem{BM}
M. Baake, R.V. Moody, \textit{Weighted Dirac combs with pure point diffraction},
J.\ reine angew.\ Math.\ (Crelle) \textbf{573}, 61--94, 2004.
\texttt{arXiv:math.MG/0203030}.


\bibitem{BF}
C. ~Berg, G. ~Forst, \textit{ Potential Theory on Locally Compact
Abelian Groups}, Springer, Berlin, 1975.

\bibitem{BCG}
I. Blech, J.  W. Cahn, D. Gratias,
\textit{ Reminiscences about a Chemistry Nobel Prize won
with metallurgy: Comments on D. Shechtman and I. A. Blech;
Metall. Trans. A, 1985, vol. 16A, pp. 1005--12},
Metallurgical and Materials Transactions A \textbf{43}, 3411--3422, 2012.


\bibitem{Dwo}
S.~Dworkin, \textit{Spectral theory and $X$-ray diffraction}, J.\ Math.\ Phys.  \textbf{34}, 2965--2967, 1993.


\bibitem{DKK}
E. H. El Abdalaoui, M. Lema\'nczyk, T. de la Rue, \textit{A dynamical point of view on the set of $\cB$-free integers},
Internat. Math. Res. Notices \textbf{2015(16)}, 7258--7286, 2015.

\bibitem{ARMA}
J . ~Gil. ~de ~Lamadrid, L.~N ~Argabright, \textit{Almost Periodic
Measures}, Memoirs Amer. Math. Soc., Vol \textbf{85}, No. 428, 1990.



\bibitem{GQ}
D. Gratias, M. Quiquandon
\textit{ Discovery of quasicrystals: The early days},
Comptes Rendus Physique \textbf{20}, 803--816, 2019.


\bibitem{Hof2}
A.~Hof, \textit{Uniform distribution and the projection method}.
In: \textit{ Quasicrystals and Discrete Geometry}, ed. J. Patera, Fields Institute Monographs \textbf{10}, AMS, Providence, RI, pp. 201--206, 1988.

\bibitem{HOF3} A. Hof, \textit{ On diffraction by aperiodic structures}, Commun. Math. Phys. \textbf{169}, 25--43, 1995.

\bibitem{IUCR} International Union of Crystallography, \textit{Report of the executive committee for 1991},   Acta Cryst. \textbf{A48}, 922--946, 1992.

\bibitem{KKL}
S. Kasjan, G.~Keller, M. Lema\'nczyk, \textit{Dynamics of $\cB$-free sets: a wiew through the window},
Internat. Math. Res. Notices \textbf{2019(9)}, 2690--2734, 2019.

\bibitem{GK}
G.~Keller, \textit{Irregular $\cB$-free Toeplitz sequences via Besicovitch’s construction of sets of multiples without density}, Monatsh Math \textbf{199}, 801--816, 2022.

\bibitem{KLRS}
G.~Keller, M. Lema\'nczyk, C.~Richard, D. Sell, \textit{On the Garden of Eden theorem for $\cB$-free subshifts},  Isr. J. Math. \textbf{251}, 567--594, 2022.

\bibitem{KR}
G.~Keller, C.~Richard, \textit{Dynamics on the graph of the torus parametrisation},   Ergod. Th. \& Dynam. Sys.  \textbf{28}, 1048--1085, 2018.

\bibitem{KR2}
G.~Keller, C.~Richard, \textit{ Periods and factors of weak model sets}, Isr. J.~Math. \textbf{229}, 85--132, 2019.

\bibitem{KRS}
G.~Keller, C.~Richard, N. Strungaru \textit{ Spectrum of weak model sets with Borel windows}, preprint, to appear in Canad. Math. Bull., 2022.
\texttt{arXiv:2107.08951}.

\bibitem{KST}
A.~Klick, N.~Strungaru, A.~Tcaciuc \textit{ On arithmetic progressions in model sets}, Discr. Comput. Geom. \textbf{67}, 930–-946, 2022.
\texttt{arXiv:2003.13860}.

\bibitem{LAG1}
J. Lagarias, \textit{ Meyer's concept of quasicrystal and
quasiregular sets},  Commun. Math. Phys. \textbf{179}, 365--376, 1996.


\bibitem{LAG}
J. ~Lagarias, \textit{ Mathematical quasicrystals and the problem
of diffraction}. In: \textit{Directions in Mathematical Quasicrystals}
eds. M. Baake and R.V Moody, CRM Monograph Series, Vol \textbf{13},
AMS, Providence, RI, pp. 61--93, 2000.


\bibitem{Last}
Y. Last, \textit{Quantum Dynamics and Decompositions ofSingular Continuous Spectra},  J. Funct. Anal. \textbf{142}, 406--445, 1996.


\bibitem{LR}
D. Lenz, C. Richard,
\textit{ Pure point diffraction and cut and project schemes for measures: the smooth case}, Math. Z. \textbf{256}, 347--378, 2007.
\texttt{math.DS/0603453}.

\bibitem{LSS} D. Lenz, T. Spindeler, N. Strungaru, \textit{Pure point diffraction and mean, Besicovitch and Weyl almost periodicity}, preprint, 2020.
\texttt{arXiv:2006.10821}.

\bibitem{LSS2}
D. Lenz, T. Spindeler, N. Strungaru, \textit{Pure point spectrum for dynamical systems and mean almost periodicity}, preprint, 2020.
\texttt{arXiv:2006.10825}.

\bibitem{LS}
D. Lenz, N. Strungaru, \textit{ Pure point spectrum for measurable
dynamical systems on locally compact Abelian groups}, J. Math. Pures Appl. \textbf{92} , 323--341, 2009.
\texttt{arXiv:0704.2498}.

\bibitem{Meyer}
Y.~Meyer, \textit{ Algebraic Numbers and  Harmonic Analysis},
North-Holland, Amsterdam, 1972.

\bibitem{MEY}
Y.~Meyer, \textit{ Quasicrystals, almost periodic patterns, mean-periodic functions
and irregular sampling}, Afr. Diaspora J. Math. \textbf{13}, 7--45, 2012.

\bibitem{MOO} R. V. Moody, \textit{ Meyer sets and their duals}. In: \textit{The Mathematics of Long-Range
Aperiodic Order}, ed. R. V. Moody, NATO ASI Series , Vol \textbf{C489},
Kluwer, Dordrecht, pp. 403--441, 1997.

\bibitem{Moody}
R.~V.~Moody, \textit{Model sets: A survey}. In: \textit{ From
Quasicrystals to More Complex Systems}, eds.\ F.\ Axel, F.\
D\'enoyer, J.\ P.\ Gazeau, EDP Sciences, Les Ulis, and
Springer, Berlin, pp.\ 145--166, 2000.
\texttt{arXiv:math.MG/0002020}.

\bibitem{MoSt}
R.V.~Moody, N.~Strungaru, \textit{Almost periodic measures and their Fourier
transforms}. In: \cite{TAO2}, pp. 173--270, 2017.

\bibitem{Ped}
G.~Pedersen,
\textit{Analysis Now}, Graduate Texts in Mathematics, Springer-Verlag,New York, 1989.

\bibitem{ReiSte}
H.~Reiter, J.D.~Stegeman,
\textit{Classical Harmonic Analysis and Locally Compact
Groups}, Clarendon Press, Oxford, 2000.

\bibitem{CRS}
C. Richard, N. Strungaru, \textit{ Pure point diffraction and Poisson summation}, Ann. H. Poincar\'e \textbf{18}, 3903--3931, 2017.
\texttt{arXiv:1512.00912}.

\bibitem{CRS2}
C. Richard, N. Strungaru, \textit{ A short guide to pure point diffraction in cut-and-project sets}, J. Phys. A: Math. Theor. \textbf{50}, no 15, 2017.
\texttt{arXiv:1606.08831}.


\bibitem{RUD}
W.~Rudin, \textit{Fourier Analysis on Groups}, Wiley, New York, 1962.


\bibitem{Martin2}
M.~Schlottmann, \textit{Generalized model sets and dynamical
systems}. In: \textit{ Directions in
Mathematical Quasicrystals}, eds. M. Baake, R.V. Moody, CRM Monogr. Ser., AMS, Providence,
RI, pp.\ 143--159, 2000.

\bibitem{She}
D. Shechtman, I. Blech, D. Gratias, J.  W. Cahn,
\textit{ Metallic phase with long-range orientational order
and no translation symmetry},
Phys.\ Rev.\ Lett.\ \textbf{53}, 183--185, 1984.


\bibitem{Sin}
B. Sing, \textit{Pisot Substitutions and Beyond}, PhD thesis (Univ. Bielefeld), 2006.

\bibitem{NS1}
N.~Strungaru, \textit{Almost periodic measures and long-range
order in Meyer sets}, Discr. Comput. Geom. \textbf{33}, 483--505, 2005.


\bibitem{NS2}
N.~Strungaru, \textit{ On the Bragg diffraction spectra of a Meyer Set}, Canad. J. Math. \textbf{65}, 675--701, 2013.
\texttt{arXiv:1003.3019}.


\bibitem{NS5}
N.~Strungaru, \textit{ On weighted Dirac combs supported inside model sets}, J. Phys. A: Math. Theor. \textbf{47},  2014.
\texttt{arXiv:1309.7947}.

\bibitem{NS11}
N.~Strungaru,
\textit{Almost periodic pure point measures}. In: \cite{TAO2}, pp. 271--342, 2017.
\texttt{arXiv:1501.00945}.

\bibitem{NS20a}
N.~Strungaru, \textit{On the Fourier analysis of measures with Meyer set support}, J. Funct. Anal. \textbf{278}, 30 pp., 2020.
\texttt{arXiv:1807.03815}

\bibitem{NS21}
N.~Strungaru, \textit{Why do Meyer sets diffract?}, extended arxiv version, 2021.
\texttt{arXiv:2101.10513}


\end{thebibliography}
\end{document}